\documentclass[reqno,oneside,12pt,draft]{amsart}

\usepackage[danish,english,french,german]{babel} 
\usepackage[utf8]{inputenc} 
\usepackage[T1]{fontenc} 
\usepackage{lmodern}
\usepackage{amsmath,amsthm,amssymb,amsfonts,mathrsfs,latexsym} 
\usepackage{enumerate}
\usepackage{graphicx} 
\usepackage{verbatim} 
\usepackage[all]{xy} 
\usepackage[pagebackref, colorlinks, linkcolor=red, citecolor=blue, urlcolor=blue, hypertexnames=true]{hyperref}  
\usepackage{multirow}
\usepackage{color}
\usepackage{tikz}
\usetikzlibrary{matrix,arrows}

\makeatletter
\newtheorem*{rep@theorem}{\rep@title}
\newcommand{\newreptheorem}[2]{%
\newenvironment{rep#1}[1]{%
 \def\rep@title{#2 \ref{##1}}%
 \begin{rep@theorem}}%
 {\end{rep@theorem}}}
\makeatother

\newtheorem{thm}{Theorem}[section]
\newtheorem{thmx}{Theorem}
\newtheorem{corx}[thmx]{Corollary}
\newreptheorem{thm}{Theorem}
\newtheorem{lem}[thm]{Lemma}
\newtheorem{prop}[thm]{Proposition}
\newtheorem{cor}[thm]{Corollary}
\newtheorem*{thm*}{Theorem}
\newtheorem*{problem*}{Problem}

\newtheorem*{claim*}{Claim}
\theoremstyle{definition}
\newtheorem{defi}[thm]{Definition}

\newtheorem{rem}[thm]{Remark}

\newenvironment{claimproof}[1]{\textit{Proof of Claim. }}{\hfill $\blacksquare$}

\newcommand{\norm}[1] {\| #1 \|}
\newcommand{\abs}[1]{\lvert#1\rvert}

\newcommand{\N}{\mathbb{N}}

\renewcommand{\epsilon}{\varepsilon}
\renewcommand{\phi}{\varphi}

\DeclareMathOperator{\Tr}{Tr}

\DeclareMathOperator{\supp}{supp}

\DeclareFontFamily{U}{mathx}{\hyphenchar\font45}
\DeclareFontShape{U}{mathx}{m}{n}{
      <5> <6> <7> <8> <9> <10>
      <10.95> <12> <14.4> <17.28> <20.74> <24.88>
      mathx10
      }{}
\DeclareSymbolFont{mathx}{U}{mathx}{m}{n}
\DeclareFontSubstitution{U}{mathx}{m}{n}
\DeclareMathAccent{\widecheck}{0}{mathx}{"71}
\DeclareMathAccent{\wideparen}{0}{mathx}{"75}

\numberwithin{equation}{section}

\begin{document}

\selectlanguage{english} 
\date{\today}

\thanks{The first and third authors were partially supported by the DGI-MINECO and European Regional Development Fund, jointly, through the grant MTM2017-83487-P.
	Also, they acknowledge support from the Generalitat de Catalunya through the grant 2017-SGR-1725. The second author was partially supported by Deutsche Forschungsgemeinschaft (SFB 878) and the Alexander von Humboldt Foundation. The third author is also supported by Beatriu de Pinos programme (BP-2017-0079). The fourth author has received funding from the European Research Council (ERC) under the European Union's Horizon 2020 research and innovation programme (grant agreement no. 677120-INDEX)}



\title{Strict comparison for $C^*$-algebras arising from Almost finite groupoids}
\author{Pere Ara}
\address{Departament de Matemàtiques, Universitat Autònoma de Barcelona, 08193 Bellaterra (Barcelona), Spain}
\email{para@mat.uab.cat}

\author{Christian Bönicke}
\address{School of Mathematics and Statistics, University of Glasgow, University Gardens, Glasgow, G12 8QQ, UK}
\email{christian.bonicke@glasgow.ac.uk}

\author{Joan Bosa}
\address{Departament de Matemàtiques, Universitat Autònoma de Barcelona, 08193 Bellaterra (Barcelona), Spain}
\email{jbosa@mat.uab.cat}

\author{Kang Li}
\address{Institute of Mathematics of the Polish Academy of Sciences, \'{S}niadeckich 8, 00-656 Warsaw, Poland}
\email{kli@impan.pl}

\begin{abstract}
	In this paper we show that for an almost finite minimal ample groupoid $G$, its reduced $\mathrm{C}^*$-algebra $C_r^*(G)$ has real rank zero and strict comparison even though $C_r^*(G)$ may \emph{not} be nuclear in general. Moreover, if we further assume $G$ being also second countable and non-elementary, then its Cuntz semigroup ${\rm Cu}(C_r^*(G))$ is almost divisible and ${\rm Cu}(C_r^*(G))$ and ${\rm Cu}(C_r^*(G)\otimes \mathcal{Z})$ are canonically order-isomorphic, where $\mathcal{Z}$ denotes the Jiang-Su algebra.
\end{abstract}

\date{\today}
\maketitle
Almost finiteness for an ample groupoid was introduced by Matui in \cite{MR2876963}. He studied their topological full groups as well as the applications of almost finiteness to the homology of étale groupoids (see \cite{MR3837599} for a survey of results in this direction). In \cite{1710.00393}, David Kerr specialised to almost finite group actions and treated them as a topological analogue of probability measure preserving hyperfinite equivalence relations, with the ultimate goal of transferring ideas from the classification of equivalence relations and von Neumann algebras to the world of (amenable) topological dynamics and $\mathrm{C}^*$-algebras.

Recently, the classification program for $\mathrm{C}^*$-algebras has culminated in the outstanding theorem that all separable, simple, unital, nuclear, $\mathcal{Z}$-stable $\mathrm{C}^*$-algebras satisfying the universal coefficient theorem (UCT) are classified by their Elliott-invariant (see \cite[Corollary~D]{MR3583354} and \cite[Corollary~D]{1901.05853}). Recall that a $C^*$-algebra is \emph{$\mathcal{Z}$-stable} if $A\otimes \mathcal{Z} \cong A$, where $\mathcal{Z}$ denotes the so-called Jiang-Su algebra. By the Toms-Winter conjecture  $\mathcal{Z}$-stability is conjecturally equivalent to strict comparison (or equivalently, almost unperforation of the Cuntz semigroup) for separable, simple, nuclear, non-elementary $C^*$-algebras. It is known that $\mathcal{Z}$-stability implies strict comparison in general and the converse is indeed the last remaining open step in the Toms-Winter conjecture (see \cite{WinterICM} for an overview and \cite{1912.04207} for the state of the art for the conjecture).

Going back to topological dynamics, David Kerr's approach in \cite{1710.00393} has seen dramatic success. He was able to show that a crossed product $C(X)\rtimes \Gamma$ associated to a free and minimal action of an (amenable) infinite group $\Gamma$ is $\mathcal{Z}$-stable provided that the action is almost finite (see \cite[Theorem~12.4]{1710.00393}). Combining this with the recent 
result in \cite[Theorem~8.1]{KerrSzabo20}, which states that every free action of a countably infinite (amenable) group with subexponential growth on a compact metrizable space with finite covering dimension is almost finite, we get a huge supply of classifiable $\mathrm{C}^*$-algebras arising from topological 
dynamics (see \cite[Theorem~8.2]{KerrSzabo20}). On the other hand, important results by Kumjian \cite{MR854149} and Reanult \cite{MR2460017} show that (twist\-ed) étale groupoids play a role in $\mathrm{C}^*$-algebras similar to the role of probability measure preserving equivalence relations play in the theory of von Neumann algebras. Moreover, Xin Li proved 
in \cite{XinLi20} that every separable, simple, unital, nuclear, $\mathcal{Z}$-stable $\mathrm{C}^*$-algebra satisfying the (UCT) has a twisted étale groupoid model. Consequently, we are led to study $\mathcal{Z}$-stability and strict comparison of groupoid $\mathrm{C}^*$-algebras.

In this article we take a step in this direction by considering the case of étale groupoids with a zero-dimensional compact unit space. Indeed, we take a slightly different route than Kerr and verify the last condition in the Toms-Winter conjecture for $C^*$-algebras arising from almost finite groupoids:

\begin{thmx}\label{TheoremA}
Let $G$ be an almost finite minimal ample groupoid with compact unit space. Then its reduced groupoid $C^*$-algebra $C_r^*(G)$ has strict comparison and real rank zero. In particular, the Cuntz semigroup ${\rm Cu}(C_r^*(G))$ is almost unperforated.
	
	If we furthermore assume that $G$ is second-countable and non-elementary\footnote{That is to say $G\ncong \mathcal{R}_n$ for any $n\in \N$, where $\mathcal{R}_n$ is the discrete full equivalence relation on $\{1,\ldots, n\}$.}, then ${\rm Cu}(C_r^*(G))$ is almost divisible and ${\rm Cu}(C_r^*(G))\cong {\rm Cu}(C_r^*(G)\otimes \mathcal{Z})$ order-isomorphic via the first factor embedding.
\end{thmx}

The class of groupoids (and their C*-algebras) under study in Theorem \ref{TheoremA} may have bizarre properties. Indeed, part of the novelty of this result is that it holds even for non-separable and non-nuclear C*-algebras. For instance, Gabor Elek constructed in \cite[Theorem~6]{1812.07511} a non-amenable minimal almost finite ample groupoid $G$ so that $C_r^*(G)$ is not nuclear. In addition, we show in Remark~\ref{non-exact} that the $C^*$-algebra $C_r^*(G)$ of an almost finite ample groupoid may not even be exact. 

In \cite{Suzuki} Suzuki develops a new strategy which in essence  
is a local version of Phillips' large subgroupoid technique (\cite{MR2134336}).
Using this method, he is able to verify that almost finite minimal groupoid  
C*-algebras have stable rank 1. Moreover, in \cite[Remark~4.3]{Suzuki} he claims that a suitably adapted strategy indeed also yields real rank zero and strict comparison for such groupoid $C^*$-algebras. In this note we carry out all the necessary intermediate steps (some  
of which might be of independent interest) in detail, as we believe
concrete proofs of these facts would be a useful contribution to the  
literature.

As mentioned before, Toms-Winter conjecture predicts that $C_r^*(G)$ in the above theorem should be $\mathcal{Z}$-stable, provided that $G$ is also assumed to be amenable. Combining Theorem~\ref{TheoremA} and Proposition~\ref{Prop:One-to-one Measures Traces} with the main theorems in \cite{MR3276157, MR3352253, 1209.3000, BBSTWW19}, we obtain the following:
\begin{corx}

Let $G$ be an amenable minimal second-countable non-elementary almost finite ample groupoid with compact unit space. Let $M(G)$ be the compact convex set of invariant positive regular Borel probability measures on $G^{(0)}$.

If the extremal boundary of $M(G)$ is compact and finite-dimensional in the weak$^*$-topology, then $C_r^*(G)$ is a separable simple unital nuclear $\mathcal{Z}$-stable $C^*$-algebra. 
\end{corx}
\begin{rem}
In a very recent article Castillejos, Evington, Tikuisis and White proved that the Toms-Winter conjecture holds among separable simple nuclear, non-elementary $C^*$-algebras which have uniform property $\Gamma$ (see \cite[Definition~2.1 and Theorem~A]{1912.04207}). Let $G$ be a non-elementary, minimal, amenable, second-countable étale groupoid with compact unit space such that the extremal boundary of $M(G)$ is compact and finite-dimensional in the weak$^*$-topology, then the reduced groupoid $C^*$-algebra $C_r^*(G)$ has uniform property $\Gamma$ if $G$ is either principal by \cite[Proposition~5.7]{1912.04207} and \cite[Lemma~4.3]{1703.10505} or almost finite by Proposition~\ref{Prop:One-to-one Measures Traces}.
\end{rem}
\textbf{Throughout the paper, all groupoids are assumed to be locally compact, Hausdorff, and their unit spaces are assumed to be compact and totally disconnected.}

\section{Preliminaries}
In this first section, we will recall some background about both ${C}^*$-algebras and groupoids. We encourage the reader to look at \cite{ABBL} for further details about these topics.

\subsection{The Cuntz semigroup and Murray-von Neumann semigroup}
Let $A$ be a $C^*$-algebra and let $\mathcal{K}$ denote the algebra of compact operators on a separable infinite-dimensional Hilbert space. Let $(A\otimes \mathcal{K})_+$ denote the set of positive elements in $A\otimes \mathcal{K}$. Given $a,b\in (A\otimes \mathcal{K})_+$, we say that $a$ is \emph{Cuntz subequivalent} to $b$ (in symbols $a\precsim b$), if there is a sequence $(v_n)$ in $A\otimes \mathcal{K}$ such that $a=\lim_n v_n bv_n^*$. We say that $a$ and $b$ are \emph{Cuntz equivalent} (in symbols $a\sim b$), if both $a\precsim b$ and $b\precsim a$. The relation $\precsim$ is clearly transitive and reflexive and $\sim$ is an equivalence relation on $(A\otimes \mathcal{K})_+$.

We define the \emph{Cuntz semigroup} of a $C^*$-algebra $A$ to be ${\rm Cu}(A)=(A\otimes \mathcal{K})_+/\sim$, and the equivalence class of $a \in (A\otimes \mathcal{K})_+$ in ${\rm Cu}(A)$ is denoted by $\langle a \rangle$. In particular,  ${\rm Cu}(A)$ is a partially ordered abelian semigroup equipped with order and addition  as:
\begin{align*}
\langle a \rangle \leq \langle b \rangle \Leftrightarrow a\precsim b,
\hspace{1,5cm}\langle a \rangle + \langle b \rangle= \langle a\oplus b \rangle,
\end{align*}
using a suitable isomorphism between $M_2(\mathcal{K})$ and $\mathcal{K}$. 

Similarly, the \emph{Murray-von Neumann semigroup} $V(A)$ of a $C^*$-algebra $A$ is defined as the set of Murray-von Neumann equivalence classes of projections in $(A\otimes \mathcal{K})$. Recall that for $p$ and $q$ projections in $(A\otimes \mathcal{K})$, we say that $p$ and $q$ are \emph{Murray-von Neumann equivalent} if there exists $v\in (A\otimes \mathcal{K})$ with $p=vv^*$ and $q=v^*v$. The class of a projection $p\in (A\otimes \mathcal{K})$ in $V(A)$ is denoted by $[p]$. We also say that $p$ is \emph{Murray-von Neumann subequivalent} to $q$ if $p$ is Murray-von Neumann equivalent to a subprojection of $q$. 
It is worth mentioning that when $A$ is a stably finite $C^*$-algebra, the natural map $V(A)\rightarrow {\rm Cu}(A)$ given by $[p]\mapsto  \langle p \rangle$ is an injective order-embedding. 
In this article we are only concerned with stably finite $C^*$-algebras. Hence, we will use this order-embedding without further mention.  
We encourage the readers to look at \cite{APT11} for further details.

\subsection{Strict comparison}
Let $T(A)$ be the tracial state space of a $C^*$-algebra $A$. Given $\tau\in T(A)$, there is a canonical extension of $\tau$ to a trace $\tau_\infty:(A\otimes \mathcal{K})_+\rightarrow [0,\infty]$. 
Abusing notation, we usually denote $\tau_\infty$ by $\tau$. The induced \emph{lower semicontinuous dimension function} $d_\tau:(A\otimes \mathcal{K})_+\rightarrow [0,\infty]$ is given by
$$
d_\tau(a):=\lim_n\tau(a^\frac{1}{n}),
$$
for $a\in (A\otimes \mathcal{K})_+$.

If $a,b\in (A\otimes \mathcal{K})_+$ satisfy $a\precsim b$, then $d_\tau(a)\leq d_\tau(b)$. Therefore, $d_\tau$ induces a well-defined, order-preserving map ${\rm Cu}(A) \rightarrow [0,\infty]$, which we also denote by $d_\tau$.

\begin{defi}
Let $A$ be a unital simple $C^*$-algebra. We say that $A$ has \emph{strict comparison} (with respect to tracial states) if for all $a,b\in (A\otimes \mathcal{K})_+$ we have $a\precsim b$ whenever $d_\tau(a)<d_\tau(b)$ for all $\tau \in T(A)$.
 \end{defi}
If a unital simple $C^*$-algebra $A$ has strict comparison (with respect to tracial states), then its Cuntz semigroup ${\rm Cu}(A)$ is \emph{almost unperforated} in the sense that whenever $\langle a \rangle, \langle b \rangle\in {\rm Cu}(A)$ satisfy $(k+1)\langle a \rangle\leq k \langle b \rangle$ for some $k\in \N$, it follows that $\langle a \rangle \leq \langle b \rangle$. If $A$ is an exact $C^*$-algebra, then every finite-valued 2-quasitrace on $A$ is a trace (see \cite{MR3241179}). Hence, the converse implication holds for all unital simple exact $C^*$-algebras (see \cite[Remark~9.2. (3)]{1711.04721}).



\subsection{Groupoids}
Given a groupoid $G$ we usually denote its unit space by $G^{(0)}$ and write $r,s:G\rightarrow G^{(0)}$ for the range and source maps, respectively. In this paper, we will only consider groupoids equipped 
with a locally compact, Hausdorff topology making all the structure maps continuous. A groupoid $G$ is called \textit{étale} if the range map, regarded as a map $r:G\rightarrow G$, is a local homeomorphism, and it is called \textit{ample} if additionally, the unit space $G^{(0)}$ is totally disconnected. Moreover, a subset $V\subseteq G$ is called \textit{bisection} if the restrictions of the source and range maps to $V$ are homeomorphisms onto their respective images. Recall that every ample groupoid $G$ admits a basis for its topology consisting of compact and open bisections.

The product of two subsets $A,B\subseteq G$ in G is given by
$$AB=\lbrace ab\in G\mid a\in A, b\in B, s(a)=r(b)\rbrace .$$ 
Whenever $B=\lbrace x\rbrace$ for a single element $x\in G^{(0)}$, we will omit the braces and just write $Ax$.

For a subset $D\subseteq G^{(0)}$, we say that the set $D$ is \textit{$G$-invariant} if for every $g\in G$ we have $r(g)\in D \Leftrightarrow s(g)\in D$, and we say that $D$ is $G${\it -full} if it satisfies that $r(GD)=G^{(0)}$. Related to that, we say that a groupoid $G$ is \textit{minimal} if there are no proper non-trivial closed $G$-invariant subsets of $G^{(0)}$. Moreover, a Borel measure $\mu$ on $G^{(0)}$ is called invariant if $\mu(s(V))=\mu(r(V))$ for every open bisection $V \subseteq G$; we will denote by $M (G)$
the compact (in the weak$^*$-topology) convex set of invariant positive regular Borel probability measures on $G^{(0)}$.

The isotropy groupoid of $G$ is the subgroupoid $Iso(G)=\lbrace g\in G\mid s(g)=r(g)\rbrace$, and we say that $G$ is \textit{principal} if $Iso(G)=G^{(0)}$. We say that $G$ is {\it topologically principal} if the set of points of $G^{(0)}$ with trivial isotropy group is dense in $G^{(0)}$.
  
  Let us finish this subsection by recalling that the reduced $C^*$-algebra associated to an étale groupoid $G$, denoted by  $C^*_r(G)$, is the completion of $C_c(G)$ by the norm coming from a single canonical regular representation of $C_c(G)$ on a Hilbert module over $C_0(G^{(0)})$ (see \cite{Re80} for further details).

\subsection{Almost finiteness} In this subsection, we recall the definition of almost finiteness and state some known properties for almost finite groupoids. 
\begin{defi}{\cite[Definition~6.2]{MR2876963}}\label{Def:AlmostFinite} Let $G$ be an ample groupoid with compact unit space.
	\begin{enumerate}
		\item We say that $K\subseteq G$ is an \textit{elementary} subgroupoid if it is a compact open principal subgroupoid of $G$ such that $K^{(0)}=G^{(0)}$.
		\item Given a compact subset $C\subseteq G$ and $\varepsilon>0$, a compact subgroupoid $K\subseteq G$ with $K^{(0)}=G^{(0)}$ is called $(C,\varepsilon)$\textit{-invariant}, if for all $x\in G^{(0)}$ we have
		$$\frac{\abs{CKx\setminus Kx}}{\abs{Kx}}<\varepsilon.$$	
		\item We say that $G$ is \textit{almost finite} if for every compact set $C\subseteq G$ and every $\varepsilon>0$, there exists a $(C,\varepsilon)$-invariant elementary subgroupoid $K\subseteq G$.
	\end{enumerate}
\end{defi}
\textbf{Throughout the paper, whenever we say that a groupoid $G$ is almost finite, we also assume that $G$ is an ample groupoid with compact unit space.}

\begin{defi}\cite[Definition 3.2]{Suzuki}\label{def:castle}
	Let $K$ be a compact groupoid. A \textit{clopen castle} for $K$ is a partition $$K^{(0)}=\bigsqcup\limits_{i=1}^n \bigsqcup\limits_{j=1}^{N_i} F_j^{(i)}$$ into non-empty clopen subsets such that the following conditions hold:
	\begin{enumerate}
		\item For each $1\leq i\leq n$ and $1\leq j,k\leq N_i$ there exists a unique compact open bisection $V_{j,k}^{(i)}$ of $K$ such that $s(V_{j,k}^{(i)})=F_k^{(i)}$ and $r(V_{j,k}^{(i)})=F_j^{(i)}$.
		\item $$K=\bigsqcup\limits_{i=1}^n \bigsqcup\limits_{1\leq j,k\leq N_i} V_{j,k}^{(i)}.$$
	\end{enumerate}
	The pair $(F_1^{(i)},\lbrace V_{j,k}^{(i)}\mid 1\leq j,k\leq N_i\rbrace)$ is called the $i$-th tower of the castle and the sets $F_j^{(i)}$ are called the levels of the $i$-th tower.
\end{defi}
\begin{rem}
	Note that the uniqueness of the bisections in $(2)$ above has an important consequence: If $\theta_{j,k}^{(i)}:F_k^{(i)}\rightarrow F_j^{(i)}$ denotes the partial homeomorphism corresponding to the bisection $V_{j,k}^{(i)}$, i.e. $\theta_{j,k}^{(i)}=r\circ (s_{\mid V_{j,k}^{(i)}})^{-1}$, then we have $(\theta_{j,k}^{(i)})^{-1}=\theta_{k,j}^{(i)}$,  $\theta_{j,k}^{(i)}\circ \theta_{k,l}^{(i)}=\theta_{j,l}^{(i)}$, and $\theta_{j,j}^{(i)}=id_{F_j^{(i)}}$.
\end{rem}

As already mentioned in \cite{Suzuki}, every compact ample principal groupoid always admits a clopen castle by \cite[Lemma~4.7]{MR2876963}. It follows that Definition~\ref{Def:AlmostFinite} is equivalent to the definition of almost finiteness given in \cite[Definition 3.6]{Suzuki} by Suzuki. Due to this fact, we will be using both equivalent notions of almost finiteness without further notice.

Finally, let us list some facts about almost finite groupoids that will be used in the sequel:

\begin{enumerate}
	\item If $G$ is an almost finite groupoid, it follows that $M(G)\neq\emptyset$ by \cite[Lemma~3.9]{Suzuki}. In particular, its extreme boundary $\partial_e M(G)$ is non-empty as well. 
	\item If $G$ is almost finite and minimal, then $G$ is topologically principal by \cite[Lemma 3.10]{Suzuki}. 
	\item Let $G$ be an almost finite groupoid and $A,B$ be compact open subsets of $G^{(0)}$. If $\mu(A)<\mu(B)$ for all $\mu \in M(G)$, then $A\precsim B$ by \cite[Lemma~3.7]{ABBL}, where $A\precsim B$ means $A$ is \textit{dynamically subequivalent} to $B$ in the sense that there exist finitely many compact open bisections $V_1 , \ldots , V_n$ of $G$ such that $A =\cup_{i=1}^n s(V_i)$ and the sets $\{r(V_i)\}_{i=1}^n$ are pairwise disjoint subsets of $B$. In particular, $1_A$ is Murray-von Neumann subequivalent to $1_B$ in $C_r^*(G)$, where $1_A$ denotes the characteristic function with support $A$.
\end{enumerate}
\section{$\mathrm{C}^*$-algebras of almost finite groupoids} 
This section is the main part of the paper. Here we verify two important facts mentioned without proof in \cite[Remark 4.3]{Suzuki} by Suzuki: $\mathrm{C}^*$-algebras of minimal almost finite groupoids have real rank zero and strict comparison. These are build upon local versions of
 results in \cite{MR2134336}, but there are some subtle differences which we expose below. 

 Let us begin by identifying the tracial states on $C_r^*(G)$, which might be of independent interest.
 It is well-known that for a principal \'etale groupoid $G$, then the tracial states on $C_r^*(G)$ are in a one-to-one correspondence with the invariant probability measures on $G^{(0)}$.
 Since an almost finite groupoid $G$ is in some sense locally approximated by principal groupoids, it might not come as a surprise that the one-to-one correspondence persists in this more generenal setting.

\begin{lem}\label{Lemma:TracesDeterminedByRestriction}
	Let $G$ be an almost finite groupoid and $\tau$ be a tracial state on $C_r^*(G)$. Then  $$\tau=\tau_{\mid C(G^{(0)})}\circ E,$$ where $E:C_r^*(G)\rightarrow C(G^{(0)})$ is the canonical conditional expectation.
\end{lem}
\begin{proof}
	For convenience let $\tau':=\tau_{\mid C(G^{(0)})}\circ E$. It is enough to show that for every fixed $f\in C_c(G)$, we have $\abs{\tau(f^*f)-\tau'(f^*f)}<\varepsilon$ for any $\varepsilon>0$, since the linear span of elements of the form $f^*f$ is dense in $C_r^*(G)$.
	We may assume that $\norm{f}\leq 1$ as well. As $supp(f^*f)$ is compact we can find compact open bisections $V_1,\ldots,V_N$ in $G$ such that $supp(f^*f)\subseteq \bigcup_{i=1}^N V_i$. Let $V$ be the (compact and open) union of the $V_i$.
	Applying almost finiteness of $G$ now, we can find a $(V\cup V^{-1},\frac{\varepsilon}{2N})$-invariant elementary subgroupoid $K$ of $G$. Clearly, $K$ is also $(V_i\cup V_i^{-1},\frac{\varepsilon}{2N})$ invariant for every $1\leq i\leq N$.
	The restrictions of $\tau$ and $\tau'$ to the subalgebra $C(G^{(0)})$ define the same $G$-invariant probability measure $\mu\in M(G)$. Since $K$ is compact open in $G$ with $K^{(0)}=G^{(0)}$, we can also view $\mu$ as an element in $M(K)$.
	By \cite[Lemma~3.8]{Suzuki} we have $\abs{\mu(r(V_i\setminus K))}<\frac{\varepsilon}{2N}$ for every $1\leq i\leq N$. Hence, we get
	$$\abs{\mu(r(V\setminus K))}\leq \sum\limits_{i=1}^N \abs{\mu(r(V_i\setminus K))}<\frac{\varepsilon}{2}.$$
	In other words, if $p:=\chi_{r(V\setminus K)}$ denotes the characteristic function of $r(V\setminus K)$, then 
	$$\tau(p)=\tau'(p)<\frac{\varepsilon}{2}.$$
	We can now follow the arguments in \cite[Lemma~2.10]{MR2134336} to get the result. For the convenience of the reader we reproduce the argument here:
	First, note that from $((1-p)f^*f)(g)=(1-p)(r(g))f^*f(g)$ and the definition of $p$, it follows that $(1-p)f^*f\in C(K)$. By taking the adjoint, we also get $f^*f(1-p)\in C(K)$. Since $p\in C(K)$, it follows that
	$$f^*f-pf^*fp=(1-p)f^*f+pf^*f(1-p)\in C(K).$$
Since $K$ is a principal groupoid, $\tau$ and $\tau'$ coincide on the $\mathrm{C}^*$-subalgebra $C_r^*(K)\subseteq C_r^*(G)$ (see for example \cite[Lemma~4.3]{1703.10505}). In particular, we get
	$$\tau(f^*f-pf^*fp)=\tau'(f^*f-pf^*fp).$$
	On the other hand, it follows from $pf^*fp\leq \norm{f}^2p\leq p$, that we have
	$0\leq \tau(pf^*fp)\leq \tau(p)<\frac{\varepsilon}{2}$ and similarly 
	$0\leq \tau'(pf^*fp)\leq \tau'(p)<\frac{\varepsilon}{2}$.
	Combining these facts we arrive at
	$$\abs{\tau(f^*f)-\tau'(f^*f)}=\abs{\tau(pf^*fp)-\tau'(pf^*fp)}<\varepsilon,$$
	as desired.
\end{proof}
Recall that $M(G)$ denotes the compact (in the weak$^*$-topology) convex set of invariant positive regular Borel probability measures on $G^{(0)}$, and $T(C_r^*(G))$ denotes the tracial state space of $C_r^*(G)$.

\begin{prop}\label{Prop:One-to-one Measures Traces}
	Let $G$ be an almost finite groupoid. Then the canonical map $T(C_r^*(G))\rightarrow M(G)$ is an affine homeomorphism. In particular, we can also identify their extreme boundaries $\partial_eT(C_r^*(G))=\partial_e M(G)$, which are non-empty.
\end{prop}
\begin{proof}
	It is well-known that this map is affine, continuous, and surjective. Injectivity now follows from Lemma \ref{Lemma:TracesDeterminedByRestriction}. By the affineness we also 
	have that $\partial_eT(C_r^*(G))=\partial_e M(G)$, which are non-empty as $M(G)\neq \emptyset$. 
\end{proof}

 Let us now focus on the proofs of real rank zero and strict comparison. For many of the intermediate steps in the proof, we only need the hypothesis that
 $G^{(0)}$ admits an invariant measure with full support (i.e., $\mu \in M(G)$ such that $\text{supp}(\mu) = G^{(0)}$). Clearly, every measure in $M(G)\ne \emptyset$ has full support for a minimal almost finite groupoid $G$ (see \cite[Lemma 6.8]{MR2876963}). 
 


 

\begin{lem}\label{Lemma:CutdownBySmallProjections}
	Let $G$ be an almost finite groupoid such that $G^{(0)}$ admits a full-supported invariant measure. For every finite subset $F\subseteq C_c(G)$ 
	and every $\varepsilon>0$, there exists an elementary subgroupoid $K\subseteq G$ and a compact open subset $W\subseteq G^{(0)}$ such that 
	if $p:=\chi_W$ is the characteristic function on $W$, then the following are satisfied:
	\begin{enumerate}
		\item $r(supp(f)\cap (G\setminus K))\cup s(supp(f)\cap (G\setminus K))\subseteq W$ for all $f\in F$,
		\item $\norm{(1-p)f(1-p)}>\norm{f}-\varepsilon$ for all $f\in F$, and
		\item $\tau(p)<\varepsilon$ for all $\tau\in T(C_r^*(G))$.
	\end{enumerate}
\end{lem}
\begin{proof}
By \cite[Corollary 2.4]{KS02}, the condition about the existence of a full-supported invariant measure $\nu$ guarantees that the associated regular representation $\pi:C_r^*(G)\rightarrow B(L^2(G,\nu))$ is 
injective.
 
	 Using that, write $F=\lbrace f_1,\ldots,f_k\rbrace$, and choose functions $\xi_1,\ldots, \xi_k,\eta_1,\ldots,\eta_k\in C_c(G)$ such that $\norm{\xi_i}=\norm{\eta_i}=1$ and $\abs{\langle \pi(f_i)\xi_i,\eta_i\rangle}>\norm{f_i}-\varepsilon$ for all $1\leq i\leq k$.
	 Consider the compact set $$C:=\bigcup\limits_{i=1}^k supp(f_i)\cup supp(f_i^*)\cup supp(\xi_i)\cup supp(\eta_i).$$
	 Since $G$ is ample, we can cover $C$ by finitely many compact open bisections $V_1,\ldots, V_l$ and we let $V:=V_1\cup\cdots\cup V_l$. Let $0<\delta<\varepsilon$ to be determined. As $G$ is almost finite, we can find a $(V\cup V^{-1},\frac{\delta}{2l})$-invariant elementary subgroupoid $K\subseteq G$. Let $W:=r(V\setminus K)\cup s(V\setminus K)$ (which depends on the choice of $\delta$). Then $(1)$ is clearly satisfied by $W$.
	 Moreover, if $\tau\in T(C_r^*(G))$, then there exists a $\mu\in M(G)$ such that $\tau(\chi_A)=\mu(A)$ for every compact open subset $A\subseteq G^{(0)}$.
	 By \cite[Lemma~3.8]{Suzuki} we have $\mu(r(V_i\setminus K))<\frac{\delta}{2l}$ and $\mu(s(V_i\setminus K))<\frac{\delta}{2l}$ for all $1\leq i\leq l$, and hence
	 $$\tau(p)=\mu(W)\leq \sum\limits_{i=1}^l \mu(r(V_i\setminus K))+\mu(s(V_i\setminus K))<\delta<\varepsilon.$$
	 It remains to check $(2)$:
	 Let $R :=\max_{1\leq i\leq l}\sup_{x\in G^{(0)}}\sum_{g\in G^x} \abs{\xi_i(g)}^2$.
	 Then we have
	 \begin{align*}	
	 \norm{\pi(1-p)\xi_i-\xi_i}^2&=\norm{\pi(p)\xi_i}^2\\
	 &=\langle \pi(p)\xi_i,\xi_i\rangle\\
	 &=\int\limits_{G^{(0)}}\sum\limits_{g\in G^x}p(r(g))\xi_i(g)\overline{\xi_i(g)}d\nu(x)\\
	 &=\int\limits_W \sum\limits_{g\in G^x}\abs{\xi_i(g)}^2d\nu(x) \le \nu (W)R < R \delta.
	 \end{align*}
	Similarly, $\norm{\pi(1-p)\eta_i-\eta_i}<R' \delta$ with $R'$ chosen as $R$, but with the $\xi_i$ replaced by $\eta_i$.
	 Using this and that $\abs{\langle \pi(f_i)\xi_i,\eta_i\rangle}>\norm{f_i}-\varepsilon$, we can choose $\delta$ (and hence $W$ and $p$) so small, such that
	 $$\abs{\langle \pi(f_i)\pi(1-p)\xi_i,\pi(1-p)\eta_i\rangle}>\norm{f_i}-\varepsilon$$
	 for $1\leq i\leq k$.
	 Hence, we get
	 \begin{align*}
	 \norm{(1-p)f_i(1-p)}&=\norm{\pi((1-p)f_i(1-p))}\\
	 &\geq \abs{\langle\pi(f_i)\pi(1-p)\xi_i,\pi(1-p)\eta_i\rangle}\\
	 &>\norm{f_i}-\varepsilon,
	 \end{align*}
	 as desired.
\end{proof}

The following lemma is a local version of \cite[Lemma~3.3]{MR2134336} for finite sets of projections. The proof follows almost verbatim to \cite[Lemma~3.3]{MR2134336}, just using Lemma \ref{Lemma:CutdownBySmallProjections} instead of \cite[Lemma~3.1]{MR2134336}. We include the proof for completeness.
\begin{lem}\label{Lemma:SubprojectionsInElementarySubgroupoid}
	Let $G$ be an almost finite groupoid such that $G^{(0)}$ admits a full-supported invariant measure. Then for every finite set of 
	projections $E=\lbrace e_1,\ldots,e_n\rbrace\subseteq C_r^*(G)$ and $\varepsilon>0$, there exists an elementary subgroupoid 
	$K\subseteq G$ and projections $q_1,\ldots,q_n\in C_r^*(K)$ such that $q_i\precsim e_i$ and $\tau(e_i)-\tau(q_i)<\varepsilon$ for all $\tau\in T(C_r^*(G))$.
\end{lem}

\begin{proof} Without loss of generality we may assume that $\varepsilon<6$.
	Choose $\delta_0>0$ such that whenever $A$ is a $\mathrm{C}^*$-algebra and $p_1,p_2\in A$ are projections with $\norm{p_1p_2-p_2}<\delta_0$, then $p_2\precsim p_1$.
	Let $f:[0,\infty)\rightarrow [0,1]$ be the continuous function given by
	$f(t)=\frac{6t}{\varepsilon}$ for $0\leq t\leq \frac{\varepsilon}{6}$ and $1$ otherwise.
	Now choose $\delta>0$ such that whenever $A$ is a $\mathrm{C}^*$-algebra and $a_1,a_2\in A$ are positive elements with $\norm{a_1},\norm{a_2}\leq 1$ and $\norm{a_1-a_2}<\delta$, then $\norm{f(a_1)-f(a_2)}<\frac{\delta_0}{2}$.

	Since $C_c(G)$ is dense in $C_r^*(G)$, there exist selfadjoint elements $d_1,\ldots, d_n\in C_c(G)$ with $\norm{d_i}\leq 1$ and
	$$\norm{e_i-d_i}<\min(\frac{\delta}{2},\frac{\varepsilon}{6}),\  1\leq i\leq n.$$
	Now apply Lemma \ref{Lemma:CutdownBySmallProjections} to $F=\lbrace d_1,\ldots, d_n\rbrace$ and $\varepsilon>0$ to obtain a projection $p=\chi_W\in C(G^{(0)})\subseteq C_r^*(G)$ and an elementary subgroupoid $K\subseteq G$, 
	such that $r(supp(d_i)\cap (G\setminus K))\cup s(supp(d_i)\cap (G\setminus K))\subseteq W$ for $i=1,\dots , n$, and $\tau(p)<\varepsilon /6$ for all $\tau \in T(C_r^*(G))$.
	Then, we have $$((1-p)d_i)(g)=\sum\limits_{h\in G^{r(g)}}(1-p)(h)d_i(h^{-1}g)=(1-p)(r(g))d_i(g),$$ which can only be nonzero if $g\in K$. Hence, we have $(1-p)d_i\in C_r^*(K)$ and $d_i(1-p)=((1-p)d_i)^*\in C_r^*(K)$.
	For every $\tau\in T(C_r^*(G))$, we have that $\tau(pe_i(1-p))=0$. Hence, $$\tau((1-p)e_i(1-p))=\tau(e_i)-\tau(e_ip)\geq \tau(e_i)-\tau(p)>\tau(e_i)-\frac{\varepsilon}{6}.$$
	Moreover, using that $\norm{d_i^2-e_i^2}\leq \norm{d_i^2-d_ie_i}+\norm{d_ie_i-e_i^2}\leq \frac{\varepsilon}{3}$ we obtain that
	$$\tau(d_i(1-p)d_i)=\tau((1-p)d_i^2(1-p))>\tau(e_i)-\frac{\varepsilon}{2}.$$
Also, each $d_i(1-p)d_i$ is a positive element in $C_r^*(K)$.
Let $g,h:[0,\infty)\rightarrow [0,1]$ be given by
\begin{equation*}
   g(t)=\left\{
   \begin{array}{ll}
      0 & 0\leq t\leq \frac{\varepsilon}{6} \\
      6\varepsilon^{-1}t-1 & \frac{\varepsilon}{6}\leq t\leq \frac{\varepsilon}{3}\\
      1 & \frac{\varepsilon}{3}\leq t
 \end{array}\right.  
    \text{ and }
     h(t)=\left\{
\begin{array}{ll}
      t & 0\leq t\leq \frac{\varepsilon}{6} \\
      \frac{\varepsilon}{6} & \frac{\varepsilon}{6}\leq t
 \end{array}
\right.
\end{equation*}
	
	Then put $a_i:=f(d_i(1-p)d_i)$, $b_i:=g(d_i(1-p)d_i)$, and $c_i:=h(d_i(1-p)d_i)$. It follows that $a_i,b_i,c_i\in C_r^*(K)$ are positive elements for all $1\leq i\leq n$.
	Moreover, we have the following relations:
	$a_ib_i=b_i$, $b_i+c_i\geq d_i(1-p)d_i$, $\norm{a_i}\leq 1$, $\norm{b_i}\leq 1$ and $\norm{c_i}\leq \frac{\varepsilon}{6}$.
	In particular we have $\tau(c_i)\leq\frac{\varepsilon}{6}$ for every $\tau\in T(C_r^*(G))$, whence
	$$\tau(b_i)=\tau(b_i+c_i)-\tau(c_i)\geq \tau(d_i(1-p)d_i)-\frac{\varepsilon}{6}>\tau(e_i)-\frac{2\varepsilon}{3}.$$
	Now use the fact that $C_r^*(K)$ is an AF-algebra to apply \cite[Lemma~3.2]{MR2134336}, which gives us projections $q_i\in \overline{b_iC_r^*(K)b_i}$ such that $a_iq_i=q_i$ and $\norm{q_ib_i-b_i}<\frac{\varepsilon}{6}$. Then $\norm{q_ib_iq_i-b_i}<\frac{\varepsilon}{3}$. So for every $\tau\in T(C_r^*(G))$ we have $$\tau(q_i)\geq\tau(q_ib_iq_i)>\tau(b_i)-\frac{\varepsilon}{3}>\tau(e_i)-\varepsilon,$$
	which is equivalent to $\tau(e_i)-\tau(q_i)<\varepsilon$.
	It remains to show, that $q_i\precsim e_i$:
	Since $\norm{d_i}\leq 1$, we have
	\begin{align*}
\norm{d_i(1-p)d_i-e_i(1-p)e_i}&<\norm{d_i(1-p)d_i-e_i(1-p)d_i} +\norm{e_i(1-p)d_i-e_i(1-p)e_i}\\
&\leq 2\norm{d_i-e_i}\\
&<\delta.
	\end{align*}
	The choice of $\delta$ then yields $\norm{a_i-f(e_i(1-p)e_i)}<\frac{\delta_0}{2}$.
	Using the equality $e_if(e_i(1-p)e_i)=f(e_i(1-p)e_i)$, we obtain $\norm{e_ia_i-a_i}<\delta_0$.
	Since $a_iq_i=q_i$, we also have $\norm{e_iq_i-q_i}=\norm{e_ia_iq_i-a_iq_i}\leq\norm{e_ia_i-a_i}\norm{q_i}<\delta_0$. From the choice of $\delta_0$ we conclude that $q_i\precsim e_i$ as desired.
\end{proof}
The following Lemma is the special tool needed to show Theorem \ref{Theorem:Traces detect positive elements in K_0}.
\begin{lem}\label{Lemma:Projections in compact groupoid algebras}
	Let $K$ be an elementary groupoid. Then for each projection $p\in C_r^*(K)$ there exists a projection $q\in C(K^{(0)})$ such that $p\sim q$.
\end{lem}
\begin{proof}
	We know that $C_r^*(K)\cong \bigoplus_{i=1}^m M_{n_i}(C(A_i))$ for some $n_i\in \N$ and pairwise disjoint clopen subsets $A_1,\ldots, A_m\subseteq K^{(0)}$. Hence, it is enough to prove the claim for an algebra of the form $M_n(C(X))$ for a compact and totally disconnected Hausdorff space $X$. So let $p\in M_n(C(X))$ be a projection. We may assume that $p\neq 0$, otherwise there is nothing to prove. Then $x\mapsto\Tr(p(x))$ is an integer valued continuous function on $X$. Using continuity and the fact that $X$ is compact and totally disconnected, we can find $r\in \N$, a partition $X=X_1\sqcup\ldots\sqcup X_r$ of $X$ by clopen subsets, and $0<n_1<\ldots< n_r\in\N$ such that $\Tr(p(x))=n_i$ for all $x\in X_i$. Note, that we must have $n_r\leq n$.
	For each $1\leq i\leq r$, let $\chi_i\in C(X)$ denote the characteristic function on $X_i$. Set $n_0:=0$ and $n_{r+1}:=n$ to make the following definition consistent: for each $1\leq i\leq r$ let $q_i\in M_{n_i-n_{i-1}}(C(X))$ be the diagonal matrix
	$$q_i:=\begin{pmatrix}
	\sum_{j=i}^r \chi_j &  & 0\\
	&\ddots &  \\
	0& & \sum_{j=i}^r \chi_j
	\end{pmatrix}.$$
	Each $q_i$ is a projection, since the characteristic functions $\chi_j$ are pairwise orthogonal.
	Define $q:=diag(q_1,\ldots, q_r,0)\in M_n(C(X))$. Then $q$ is a projection and $\Tr(q(x))=\Tr(p(x))$ for all $x\in X$. Since $X$ is totally disconnected, the result follows from \cite[Excercise 3.4]{zbMATH01541843}.
\end{proof}

\begin{thm}\label{Theorem:Traces detect positive elements in K_0}
	Let $G$ be an almost finite groupoid such that $G^{(0)}$ admits a full-supported invariant measure.
	If $x\in K_0(C_r^*(G))$ satisfies $\tau_*(x)> 0$ for all $\tau\in T(C_r^*(G))$, then there exists a projection $e\in M_\infty(C_r^*(G))$ such that $x=[e]$.
\end{thm}
\begin{proof}
	Write $x=[q]-[p]$ for two projections $p,q\in M_n(C_r^*(G))$ for some large enough $n\in\N$. Replacing $G$ by $G\times \lbrace1,\ldots, n\rbrace^2$ and using that $C_r^*(G\times \lbrace 1,\ldots, n\rbrace^2)\cong M_n(C_r^*(G))$, we may assume that $p,q\in C_r^*(G)$.
	Since $T(C_r^*(G))$ is weak-* compact, there exists $\varepsilon>0$ such that $\tau(q)-\tau(p)=\tau_*(x)>\varepsilon$ for all $\tau\in T(C_r^*(G))$.
	Now we apply Lemma \ref{Lemma:SubprojectionsInElementarySubgroupoid} to $E=\lbrace q,1-p\rbrace$ to obtain an elementary subgroupoid $K\subseteq G$ and projections $q_0,f_0\in C_r^*(K)$ such that $q_0\precsim q$ and $f_0\precsim 1-p$ and $\tau(q)-\tau(q_0)<\frac{\varepsilon}{3}$ and $\tau(1-p)-\tau(f_0)<\frac{\varepsilon}{3}$. Combining these three inequalities we get
	$$\tau(q_0)-\tau(1-f_0)>\frac{\varepsilon}{3}>0\ \forall \tau\in T(C_r^*(G)).$$
	Now since $K$ is an elementary subgroupoid of $G$, we can invoke Lemma \ref{Lemma:Projections in compact groupoid algebras} to find projections $q_1,f_1\in C(G^{(0)})$ such that $q_1\sim q_0\precsim q$ and $f_1\sim f_0\precsim 1-p$.
	Hence, $$\tau(q_1)-\tau(1-f_1)>0\ \forall \tau\in T(C_r^*(G)).$$
	By Proposition \ref{Prop:One-to-one Measures Traces} every trace corresponds to a $G$-invariant measure and vice versa. Since $q_1,f_1$ must be the characteristic functions 
	of some clopen subsets of $G^{(0)}$, it follows from \cite[Lemma 3.7]{ABBL} that $1-f_1$ is Murray-von Neumann subequivalent to $q_1$. Let $q_2\in C_r^*(G)$ be a projection such that $1-f_1\sim q_2\leq q_1$.
	Since $q_1\precsim q$, there exists a projection $q'\in C_r^*(G)$ such that $q_1\sim q'\leq q$ and since $f_1\precsim 1-p$ there exists $f'\in C_r^*(G)$ such that $f_1\sim f'\leq 1-p$. Then
	\begin{align*}
	x&=[q]-[p]=([q]-[q_1])+([q_1]-[q_2])+([q_2]-[p])\\
	& = [q-q']+[q_1-q_2]+[1-p-f']>0,
	\end{align*}
	which concludes the proof.
\end{proof}

As an easy application of Theorem \ref{Theorem:Traces detect positive elements in K_0} and the main theorem in \cite{Suzuki}, we deduce the following corollary. Recall that a C*-algebra $A$ has \emph{comparison of projections} if, 
for projections $p,q\in M_\infty(A)$, we have $p\precsim q$ whenever $\tau(p)<\tau (q)$ for all $\tau\in T(A).$

\begin{cor}\label{cor:ComparisonProjections}
	If $G$ is a minimal almost finite groupoid, then $C_r^*(G)$ has comparison of projections.
\end{cor}

\begin{proof}
If  $\tau(p)<\tau(q)$ for all $\tau\in T(C_r^*(G))$, then by Theorem 2.6 we have $[q]-[p]=[e]$ in $K_0(C_r^*(G))$ for some projection $e$. In other words, $[q]=[p\oplus e]$. 
Since $C_r^*(G)$ has stable rank one by \cite[Main Theorem]{Suzuki}, we have that $q$ is MvN-equivalent to $p\oplus e\geq p$, which concludes the proof.
\end{proof}

Let us now turn our attention to the real rank of $C_r^*(G)$. We need the following technical result inspired by \cite[Lemma~4.1]{MR2134336}.
\begin{lem}\label{Lemma:EnoughRoom}
	Let $G$ be an almost finite groupoid.
	For every finite subset $F\subseteq C_c(G)$ and $n\in\N$, there exist an elementary subgroupoid $K\subseteq G$ and a clopen subset $W\subseteq G^{(0)}$, such that for $p:=\chi_W\in C(G^{(0)})$ we have:
	\begin{enumerate}
		\item $f(1-p)$ and $(1-p)f$ are in $C_c(K)$ for all $f\in F$, and
		\item There exist $n$ mutually orthogonal projections $p_1,\ldots, p_n\in C(G^{(0)})$, such that $p_i\sim p$ in $C_r^*(G)$ for all $1\leq i\leq n$.
	\end{enumerate}	
\end{lem}

\begin{proof}
Let $F\subseteq C_c(G)$ be a finite subset and $n\in\N$. Consider the compact set $C:=\bigcup_{f\in F} \supp(f)\cup\supp(f^*)$. Find compact open bisections $V_1,\ldots, V_l$ such that $C\subseteq \bigcup_{i=1}^l V_i=:V$. Then we can use almost finiteness of $G$ to find a $(V\cup V^{-1},\frac{1}{2(n+1)l})$-invariant elementary subgroupoid $K\subseteq G$. Let $W:=r(V\setminus K)\cup s(V\setminus K)$ and $p:=\chi_W\in C_r^*(G)$. Then $p$ satisfies $(1)$, since for all $f\in F$ we can compute
$$((1-p) f)(g)=\sum\limits_{h\in G^{r(g)}}(1-p)(h)f(h^{-1}g)=(1-p)(r(g))f(g),$$
and the latter quantity can only be non-zero if $g\in K$ by the definition of $p$. Similar reasoning yields $f(1-p)\in C_c(K)$.

We now aim to show that $p$ also satisfies $(2)$. To this end we first show the following intermediate claim, which basically says, that in any given tower of a castle for $K$ that intersects $W$, we have enough levels to allow for at least $n$ pairwise disjoint copies of $W$ all equivalent in the dynamical sense to $W$. 

Before that, recall that by the same arguments as in the proof of Lemma \ref{Lemma:CutdownBySmallProjections} we have
\begin{equation}\label{Equation:W is small in measure}
\mu(W)<\frac{1}{n+1} \text{ for all } \mu\in M(G).
\end{equation}

\begin{claim*} There exists $0<\varepsilon<\frac{1}{n+1}$, a compact subset $L\subseteq G$ and a $(L,\varepsilon)$-invariant elementary subgroupoid $K'\subseteq G$ admitting a clopen castle
	 $$G^{(0)}=\bigsqcup_{i=1}^N\bigsqcup_{j=1}^{N_i}F_j^{(i)},\ \ K'=\bigsqcup\limits_{i=1}^N \bigsqcup\limits_{l,k=1}^{N_i} V_{k,l}^{(i)},$$
such that for all $1\leq i\leq N$ we have 
\begin{equation}\label{Equation:EnoughRoom}
N_i>(n+1)\cdot \big|\lbrace j\mid F_j^{(i)}\cap W\neq \emptyset\rbrace\big|.
\end{equation}
\end{claim*}
\begin{claimproof}

Suppose the claim is not true. Using almost finiteness, for every $0<\varepsilon<\frac{1}{n+1}$ and compact subset $L\subseteq G$, there exists a $m:=(L,\varepsilon)$-invariant elementary subgroupoid $K_m\subseteq G$ admitting a clopen castle. By refining the tower-decomposition according to \cite[Lemma~3.4]{ABBL}, we may as well assume that every level of every tower of the castle is either contained in or disjoint from $W$.
Since we assumed that the claim is not true, in each such clopen castle there must be at least one tower for which the inequality \ref{Equation:EnoughRoom} does not hold. Denoting the mentioned tower (and levels) by  $\mathcal F_m :=( F^{(i_m)}_{j},\theta^{(i_m)}_{j,k})_{1\leq j,k\leq N_m}$, let $x_m\in F^{(i_m)}_1$ and define the associated probability measure on $B\subseteq G^{(0)}$ by

$$\mu_m(B)=\frac{1}{N_m}\sum\limits_{j=1}^{N_m} \delta_{x_m}(\theta_{1,j}^{(i_m)}(B\cap F^{(i_m)}_{j})).$$

Then, using that inequality \ref{Equation:EnoughRoom} does not hold, for all $m$ we have that
$$\mu_m(W)\geq \sum_{j=1}^{N_m}\mu_m(F^{(i_m)}_{j}\cap W)=\frac{\big |\lbrace j\mid F^{(i_m)}_{j}\cap W\neq\emptyset\rbrace\big|}{N_m}\geq \frac{1}{n+1}.$$

Then, it can be verified that any weak-$\ast$-cluster point of the net $(\mu_m)_m$ is a $G$-invariant probability measure on $G^{(0)}$ (see the proof of \cite[Lemma~3.7]{ABBL} for more details). If $\mu\in M(G)$ is one of those, it also satisfies $\mu(W)\geq \frac{1}{n+1}$; thus, it contradicts the inequality \ref{Equation:W is small in measure}.
\end{claimproof}

Now suppose we are given a clopen castle as in the claim with associated partial homeomorphisms $\theta^{(i)}_{k,l}$ implemented by the bisections $V_{k,l}^{(i)}$. For ease of notation, let $l_i:=\big|\lbrace j\mid F_j^{(i)}\cap W\neq \emptyset\rbrace\big|$. We can relabel the levels if necessary to assume that $W$ sits at the bottom of each tower, i.e. $F_j^{(i)}\cap W=\emptyset$ if and only if $j>l_i$ for all $1\leq i\leq N$.
Then, for each $1\leq k\leq n$ let $$W_k:=\bigcup_{i=1}^N\bigcup_{j=1}^{l_i} \theta^{(i)}_{kl_i+j,j}(F_j^{(i)}\cap W).$$
Let $p_k:=\chi_{W_k}$ be the associated characteristic function. Then the $p_k$ are obviously all pairwise orthogonal and by construction $p_k\sim_G p_0=p$ for all $k\in\lbrace 0,\ldots, n\rbrace$. In particular, the $p_k$ are all Murray-von Neumann equivalent.
\end{proof}

We can now follow the proof of \cite[Theorem~4.6]{MR2134336} by using Lemma \ref{Lemma:EnoughRoom} and Theorem \ref{Theorem:Traces detect positive elements in K_0} to get the following:
\begin{thm}\label{Thm:RR0}
	If $G$ is a minimal almost finite groupoid, then $C_r^*(G)$ has real rank zero.
\end{thm}
\begin{proof}
Let $a\in C_r^*(G)$ be a selfadjoint element with $\norm{a}\leq 1$. We want to approximate $a$ by an invertible selfadjoint element. Invoking a short density argument, we may assume that $a\in C_c(G)$. Moreover, we assume $0\in sp(a)$ since for an invertible element $a$, there is nothing to prove.
Let $\varepsilon>0$ be given. Choose continuous functions $f,g:[-1,1]\rightarrow [0,1]$ such that $g(0)=1$, $fg=g$, and $\supp(f)\subseteq (-\frac{1}{9}\varepsilon,\frac{1}{9}\varepsilon)$.
Let $$\alpha:=\inf_{\tau\in T(C_r^*(G))} \tau(g(a)).$$
Since $G$ is minimal and almost finite, $C_r^*(G)$ is simple. Hence all traces on $C_r^*(G)$ are faithful. Combining this with the facts that $g(a)$ is a nonzero positive element, and $T(C_r^*(G))$ is weak* compact, we obtain $\alpha>0$.
Find $0<\delta<\frac{1}{9}\varepsilon$ from \cite[Lemma~4.4]{MR2134336} applied to $r=1$, $g$, and $\frac{1}{4}\alpha$. Now let $m\in\N$ with $m>2/\delta$. 
By Lemma \ref{Lemma:EnoughRoom} we can find an elementary subgroupoid $K\subseteq G$ and a projection $p_0\in C(G^{(0)})$, such that $a(1-p_0),(1-p_0)a\in C_c(K)$ and such that $p_0$ is Murray-von Neumann equivalent in $C_r^*(G)$ to more than $8m\alpha^{-1}$ mutually orthogonal projections in $C(G^{(0)})$. In particular $\tau(p_0)<\frac{1}{8}\alpha m^{-1}$ for all $\tau\in T(C_r^*(G))$.

Define $b=a-p_0ap_0$. Then $b$ is a selfadjoint element of $C_c(K)$ with $\norm{b}\leq 2$. By our choice of $\delta$, and the fact that $C_r^*(K)$ is an AF algebra, we can apply \cite[Lemma~4.3]{MR2134336} 
to $b$, $p_0$ and $\frac{1}{2}\delta$ to obtain a projection $p\in C_r^*(K)$ such that $\norm{pb-bp}<\delta$, $p_0\leq p$ and $[p]\leq 2m[p_0]$ in $K_0(C_r^*(G))$. 
Now $p$ commutes with $a - b = p_0ap_0$, so also $\norm{pa-ap}<\delta$. Furthermore, because $p\in C_r^*(K)$ and $p\geq p_0$, we get $(1-p)a,a(1-p)\in C_r^*(K)$.

Define $a_0:=(1-p)a(1-p)$. For every $\tau\in T(C_r^*(G))$, we have
$$\tau(p)\leq 2m\tau(p_0)<\frac{1}{4}\alpha.$$
By the choice of $\delta$ and using \cite[Lemma~4.4]{MR2134336}, we get
$$\tau(g(a_0))>\tau(g(a))-\tau(p)-\frac{1}{4}\alpha\geq \alpha-\frac{1}{4}\alpha-\frac{1}{4}\alpha=\frac{1}{2}\alpha\ \ \forall \tau\in T(C_r^*(G)).$$
Also $f(a_0)g(a_0)=g(a_0)$, and $C_r^*(K)$ is an AF algebra, so \cite[Lemma~3.2]{MR2134336} provides a projection $q\in C_r^*(K)$ such that
$$q\in \overline{g(a_0)C_r^*(K)g(a_0)},\ \ f(a_0)q=q,\text{ and }\norm{qg(a_0)-g(a_0)}<\frac{1}{8}\alpha.$$
Therefore we have the estimate $\norm{qg(a_0)q-g(a_0)}<\frac{\alpha}{4}.$
For all $\tau\in T(C_r^*(G))$, we have $\tau(qg(a_0)q)\leq \tau(q)$ because $\norm{g(a_0)}\leq 1$. Combining this with previous estimates, it follows that
$\tau(q)>\frac{1}{4}\alpha$.
Combining this with our estimate for $p$, we get that
$$\tau(p)<\tau(q)\text{ for all }\tau\in T(C_r^*(G)).$$
It follows from Theorem \ref{Theorem:Traces detect positive elements in K_0}, that $[q]-[p]=[e]$ for some projection $e\in M_\infty(C_r^*(G))$. Since $C_r^*(G)$ has stable rank one (and thus cancellation of projections) 
by \cite[Main Theorem]{Suzuki}, we have $q\sim p+e$, which means $p\precsim q$ in $C_r^*(G)$.
Since $a_0p=pa_0=0$, we conclude that $p$ and $q$ are orthogonal.
By \cite[Lemma~4.5]{MR2134336} applied to $a_0$, $\lambda_0 = 0$, $g$, and $q$ we have
\begin{equation*}
    \norm{qa_0-a_0q}<\frac{2\varepsilon}{9} \text{  and  }\norm{qa_0q}<\frac{\varepsilon}{9}.
\end{equation*}
    Consider now $s:=1-p-q$. Then
    $$a-(sas+pap)=pa(1-p)+(1-p)ap+qa_0s+sa_0q+qa_0q.$$
Therefore, using that $qs=0$, we have
\begin{align*}
    \norm{a-(sas+pap)}&\leq 2 \norm{pa-ap}+2\norm{qa_0-a_0q}+\norm{qa_0q}\\
    & < 2\delta + \frac{4\varepsilon}{9}+\frac{\varepsilon}{9}< \frac{7\varepsilon}{9}.
\end{align*}
Now if $B=(1-s)C_r^*(G)(1-s)$, then $pap$ is a selfadjoint element in $pBp=pC_r^*(G)p$ and we have $p\precsim q=(1-s)-p=1_B-p$. Hence \cite[Lemma~8]{MR1209829} provides us with an invertible selfadjoint element $b\in B$ such that $\norm{b-pap}<\frac{\varepsilon}{9}$. Moreover, $sas=s(1-p)as\in sC_r^*(K)s$, which is an AF algebra, so there is an invertible selfadjoint element $c\in sC_r^*(K)s$ such that $\norm{c-sas}<\frac{\varepsilon}{9}$. It follows that $b+c$ is an invertible selfadjoint element in $C_r^*(G)$ such that
\begin{align*}
    \norm{a-(b+c)} & \leq \norm{a-(sas+pap)}+\norm{b-pap} + \norm{c-sas}\\
    & < \frac{7\varepsilon}{9}+\frac{\varepsilon}{9}+\frac{\varepsilon}{9}=\varepsilon,
\end{align*}
which completes the proof.
\end{proof}

Stable rank one for C*-algebras associated to minimal almost finite groupoids (\cite[Main Theorem]{Suzuki}) is a crucial ingredient in the proof of Theorem \ref{Theorem:Traces detect positive elements in K_0} and Theorem \ref*{Thm:RR0}. Notice that this strategy does not hold for general non-minimal almost finite groupoids since they usually do not have stable rank one (see e.g. \cite{LW18,MR989097,BNS} for examples and further results in this direction).

Finally, we are ready to provide a proof of the main theorem by combining the above results:

\begin{proof}[Proof of Theorem \ref{TheoremA}]
First of all, we notice that $C_r^*(G)$ is a unital simple $C^*$-algebra with stable rank one and real rank zero (see \cite[Remark~6.6]{MR2876963}, \cite[Corollary~3.14]{boenicke_li_2018}, \cite[Main Theorem]{Suzuki} 
and Theorem \ref{Thm:RR0}). Therefore, its Cuntz semigroup is ${\rm Cu}(C_r^*(G))\cong\Lambda_\sigma(V(C_r^*(G)))$ (\cite[Theorem 6.4]{abp11}), where the latter stands for the countably generated intervals 
in the projection monoid. Recall that the isomorphism is described via $\langle a\rangle\mapsto I(a):=\{[p]\in V(C_r^*(G))\mid p\in \overline{aM_\infty(C_r^*(G))a}\},$ 
and that any interval $I(a)$ has an increasing countable cofinal subset of projections $\{[ p_n ]\}$ in $V(C_r^*(G))$ such that $\langle a\rangle=\sup ([p_n]) \text{ in }{\rm Cu}(C_r^*(G))$.
	
	Let us now fix $\langle a\rangle,\langle b\rangle\in {\rm Cu}(C_r^*(G)) $ such that $d_\tau(a)< d_\tau(b)$ for all $\tau\in T(C^*_r(G))$. 
	Let $\langle a\rangle=\sup ([p_n])$ and $\langle b\rangle=\sup ([q_m])$, where all $[p_n], [q_m]$ in $V(C_r^*(G))$. 
	Given $\tau\in T(C^*_r(G))$ and $n\in \N$, it is clear by construction that $\tau(p_n)=d_\tau (p_n)<d_\tau(b)$. 
	Hence, there is $N(n,\tau)\in\mathbb N$ such that $\tau(p_n)<\tau(q_{N(n,\tau)})$. Now, using that $T(C^*_r(G))$ is compact under the weak-* topology, 
	we find $N(n)\in \N$ such that $\tau(p_n)<\tau(q_{N(n)})$ for all $\tau\in T(C^*_r(G))$. By Corollary \ref{cor:ComparisonProjections}, one obtains 
	that $[p_n]\leq [q_{N(n)}]$. As this can be done for all $n\in \mathbb N$, $C_r^*(G)$ has strict comparison.
	
For the second statement, $C_r^*(G)$ is a separable non-elementary unital simple $C^*$-algebra with stable rank one. Hence, this part follows from \cite[Corollary~8.12]{1711.04721}.
	\end{proof}
\begin{rem}\label{non-exact}
It is worth noticing that $C_r^*(G)$ in Theorem~\ref{TheoremA} may not be nuclear in general, as Gabor Elek constructed non-amenable minimal almost finite ample groupoids in \cite[Theorem~6]{1812.07511}. In \cite[Corollary~4.12]{ABBL}, the same four authors of this paper construct an almost finite ample principal non-minimal groupoid $G$ from coarse geometry such that $G$ is not even a-T-menable. 

Let us finish providing here an almost finite ample principal non-minimal groupoid $G$ such that $C^*_r(G)$ is not exact. Indeed, take $X$ to be one of the expanders from \cite[Corollary~3]{MR2990566} such that its uniform Roe algebra $C_u^*(X)$ is not ($K$)-exact. Then $Y = X\times \N$ defined as in \cite[Proposition~4.10]{ABBL} contains $X$ as a subspace by construction, and $Y$ admits tilings of arbitrary
invariance. Hence, the associated coarse groupoid $G(Y)$ is almost finite by \cite[Theorem~4.5]{ABBL}. On the other hand, $C_u^*(X)$ is a $C^*$-subalgebra of $C_r^*(G(Y))=C_u^*(Y)$. Since exactness passes to $C^*$-subalgebras, $C_r^*(G(Y))$ cannot be exact as desired.
\end{rem}

\bibliographystyle{plain}
\bibliography{References}
\end{document}